\documentclass[12pt]{amsart}
\pdfoutput=1

\usepackage[paperheight=11in, paperwidth=8.5in, left=1in, top=1.25in, right=1in, bottom=1.25in]{geometry}
\usepackage[linktocpage=true, linktoc=all, colorlinks=true, linkcolor=black, citecolor=black, filecolor=black, urlcolor=black, pagebackref=false, pdfstartpage={1}, pdfstartview={FitH}, pdftitle={}, pdfauthor={}, pdfsubject={}, pdfcreator={}, pdfproducer={}, pdfkeywords={}]{hyperref}
\usepackage{afterpage}
\usepackage{amsfonts}
\usepackage{amsmath}
\allowdisplaybreaks[4]
\usepackage{amssymb}
\usepackage{amsthm}
\usepackage{array}
\usepackage{bbm}
\usepackage[justification=centering]{caption}
\usepackage[capitalize,nameinlink,noabbrev,compress]{cleveref}

\usepackage{enumerate}
\usepackage{enumitem}
\usepackage{float}
\restylefloat{figure}
\usepackage{graphicx}
\usepackage{mathrsfs}
\usepackage[framemethod=tikz,ntheorem,xcolor]{mdframed}
\usepackage{parcolumns}
\usepackage{scalerel}
\usepackage{tabularx}
\usepackage{tikz}\pdfpageattr{/Group <</S /Transparency /I true /CS /DeviceRGB>>}
\usetikzlibrary{arrows,calc,cd,intersections,matrix,positioning,through}
\tikzset{commutative diagrams/.cd,every label/.append style = {font = \normalsize}}
\usepackage[all]{xy}
\usepackage[aligntableaux=center]{ytableau}

\numberwithin{equation}{section}
\newtheorem{thm}{Theorem}
\AfterEndEnvironment{thm}{\noindent\ignorespaces}
\numberwithin{thm}{section}

\AfterEndEnvironment{conj}{\noindent\ignorespaces}

\AfterEndEnvironment{cor}{\noindent\ignorespaces}

\AfterEndEnvironment{lem}{\noindent\ignorespaces}

\AfterEndEnvironment{prob}{\noindent\ignorespaces}

\AfterEndEnvironment{prop}{\noindent\ignorespaces}

\theoremstyle{definition}
\newtheorem{defn}[thm]{Definition}
\AfterEndEnvironment{defn}{\noindent\ignorespaces}
\newtheorem{eg_no_qed}[thm]{Example}

\AfterEndEnvironment{eg}{\noindent\ignorespaces}
\newtheorem{rmk}[thm]{Remark}
\AfterEndEnvironment{rmk}{\noindent\ignorespaces}

\theoremstyle{remark}

\AfterEndEnvironment{claim}{\noindent\ignorespaces}
\newtheorem*{claimpf_no_qed}{Proof of Claim}

\AfterEndEnvironment{claimpf}{\noindent\ignorespaces}

\def\la{\lambda}

\def\pix{\pi(x)}

\def\t{{\theta}}

\title{ The Monge--Kantorovich problem, the Schur--Horn theorem, and 
the  diffeomorphism group of the annulus}

\author{Anthony M. Bloch}
\address{Department of Mathematics, University of Michigan, Ann Arbor, 
MI 48109, USA}
\email{\href{mailto:abloch@umich.edu}{abloch@umich.edu}}
\author{Tudor S. Ratiu}
\address{School of Mathematical Sciences, Ministry of Education Key Lab on
Scientific and Engineering Computing, Shanghai Frontier Science Center
of Modern Analysis, Shanghai Jiao Tong University, Minhang District, Shanghai,
200240 China}
\email{\href{mailto:ratiu@sjtu.edu.cn}{ratiu@sjtu.edu.cn}}

\thanks{A.M.B.\ was partially supported by NSF grant
DMS-2103026, and AFOSR grants  FA 9550-22-1-0215 and FA 9550-23-1-0400. T.S.R was partially 
supported by the National Natural Science Foundation of China  grant
number 1247012587 and by the Swiss National Science Foundation grant 
NCCR SwissMAP.}

\usepackage{fancyhdr}
\begin{document}

\thispagestyle{plain}

\begin{abstract} First, we analyze the discrete Monge--Kantorovich 
problem, linking it with the minimization problem of linear 
functionals over adjoint orbits. Second, we consider its 
generalization to the setting of area preserving diffeomorphisms 
of the annulus. In both cases, we show how the problem can 
be linked to permutohedra, majorization, and to gradient 
flows with respect to a suitable metric. 
\end{abstract}

\maketitle

\section{Introduction} In this paper, we consider a geometric 
approach to analyzing a particular setting of the Monge--Kantorovich 
problem and its connection to Schur--Horn theory. 
We begin by analyzing the  finite-dimensional setting  
(as in Brezis \cite{Brezis2018}) and then we consider 
the Monge--Kantorovich problem on the diffeomorphism group 
of the annulus. This is a particular interesting case of 
the general infinite-dimensional problem, as presented 
in,  for example, Evans \cite{evans2001} and Gangbo and McCann 
\cite{gangbo1996geometry}. We relate these problems, respectively, 
to the  classical  Schur--Horn theorem,  and its infinite-dimensional 
generalization to the diffeomorphism group of the annulus 
as proved in \cite{bloch_flaschka_ratiu93}.  We also consider 
the dual problems. In addition, we relate these problems to 
the gradient flows discussed in 
\cite{bloch_brockett_ratiu92}, \cite{bloch_karp23a}, and 
\cite{bloch_flaschka_ratiu96}, casting them in a dynamical setting.

There are, of course, many other treatments and applications 
of the Monge--Kantorovich problem; see the references in the 
papers cited above and, e.g., \cite{chen2021optimal},
\cite{chen2017matrix}. For the theory applied here, this paper  
relies  on previous work of the authors and introduces  
a novel setting for the problem as well as 
a novel mathematical approach to the annulus setting which 
is close to the analysis of Brezis in finite dimensions.

\section{The finite dimensional case}

Brezis \cite{Brezis2018} introduces the following finite-dimensional setting.

Consider two sets $X,Y$ consisting of $m$ points $\{P_i\}$ and 
$\{N_i\},\,1\le i \le m$, i.e.,

\begin{equation}
X=\{P_1,P_2\dots ,P_m\}\,,\quad Y=\{N_1,N_2\dots ,N_m\}\,.
\end{equation}
Let $c:X\times Y\rightarrow\mathbb{R}$ be a smooth ``cost" function.
Brezis introduces three problems denoted $\mathbf{M}$, $\mathbf{K}$,
and $ \mathbf{D}$ (for dual): 
\begin{equation}
\mathbf{M} :=\text{min}_{\sigma\in S_m}\sum_{i=1}^m c(P_i,N_{\sigma(i)}),
\label{M}
\end{equation}
where $S_m$ is the permutation group of $\{1, \ldots, m\}$, and 
\begin{equation}
\mathbf{K} :=\text{min}_{A}\left\{\left. \sum_{i,j=1}^m a_{ij}c(P_i,N_j)\,
\right|\, A=(a_{ij})\, \text{is doubly stochastic}\right\}.
\label{K}
\end{equation}
For our purposes, we formulate the dual problem as in 
\cite{Brezis2018}:
\begin{equation}
\qquad\quad
\mathbf{D} :=\text{sup}_{\varphi:X\rightarrow\mathbb{R},\,
\psi:Y\rightarrow\mathbb{R}}\left\{\left. \sum_{i=1}^m 
(\varphi(P_i)-\psi(N_i))\, \right|\,\varphi(x)-\psi(y)\le c(x,y), 
\forall x\in X, y\in Y\right\}.
\end{equation}

Brezis proves that $\mathbf{M} =\mathbf{K} =\mathbf{D}$. We shall 
consider these equalities from the point of view of majorization 
and dynamics. It is clear, for example, that $\mathbf{M}\ge
\mathbf{K}$ since a permutation matrix (every row and column has
exactly one entry equal to 1 and all other entries equal to 0) is 
a special case of a doubly stochastic matrix (all entries are 
$\geq 0$ and the sum  of all entries in each row and each column 
is 1). We shall also show this is true in our infinite-dimensional 
setting. We also show that $\mathbf{K}\ge\mathbf{D}$ in 
both cases. Brezis also shows $\mathbf{D}\ge\mathbf{M}$ 
in the finite case completing the equalities.  We will not go through the 
details of this step here but the result applies to our setting. 

\subsection{The adjoint orbit and dynamical setting}

We can arrive also at this finite setting by considering the problem 
of minimizing $\text{Trace}(LN)$, where $L$ belongs to 
the isospectral set of skew Hermitian matrices defined by the purely
imaginary diagonal matrix $\Lambda$ and $N$ is a constant skew 
Hermitian matrix, as in \cite{bloch_brockett_ratiu92}.

More precisely we consider  
\begin{align}
\text{min}_{\Theta}\|L-N\|^2&:=\text{min}_{\Theta}
\left\langle L-N,L-N\right\rangle 
:=\text{min}_{\Theta}\left\langle \Theta^T\Lambda\Theta-N,
\Theta^T\Lambda\Theta-N\right\rangle 
\end{align}
where $\Theta$ belongs to the group of unitary matrices. This  is 
equivalent to minimizing $\text{Trace}(LN)$ and is the Monge problem 
$\mathbf{M}$ in this setting. 

We arrive at the solution by following the gradient dynamics
\begin{equation}
 \dot{L}=[L,[L,N]]
 \label{dbb}
\end{equation}
which is the gradient flow of $\text{Trace}(LN)$ on an adjoint orbit
of the unitary group $\operatorname{U}(n)$
with respect to the normal metric. 

\begin{thm}
Equation \eqref{dbb} is  the gradient flow of  ${\rm Trace}(LN)$ 
with respect to the normal metric on an adjoint orbit of 
$\operatorname{U}(n)$. For $N$ diagonal with distinct diagonal 
entries and $L$ having initial condition with distinct eigenvalues, 
there are $n!$ 
equilibria corresponding to the $n!$  diagonal matrices 
with rearranged eigenvalues. The stable equilibrium 
 is the one having 
the same ordering as the entries of $N$, 
after dividing both by $i$.
\end{thm}

We give the proof of the gradient form below. The other conclusions
follow from the fact that the flow is isospectral (see also
\cite{bloch_brockett_ratiu92} for detailed proofs). 
The gradient flow of $\text{Trace}(LN)$ on an adjoint orbit 
$\mathcal{O}$ of the unitary group endowed with the ``standard'' 
or ``normal'' metric will be computed below.  Explicitly, this 
metric is given as follows.

Let $\kappa : \mathfrak{u}(n) \times \mathfrak{u}(n)
\rightarrow  \mathbb{R}$ be the Killing form $ \kappa (AB):=
\operatorname{Trace}(AB)$. Decompose orthogonally $L \in \mathfrak{g}:= 
\mathfrak{u}(n)$ relative to the inner product $\langle~~,~~\rangle:=
-\kappa (~~,~~)$, $\mathfrak{g} = \mathfrak{g}^L
\oplus\mathfrak{g}_{L}$ where $\mathfrak{g}_{L}$ is the centralizer 
of $L$ and $\mathfrak{g}^L=\text {Im}\operatorname{ad}_L$, where
$ \operatorname{ad}_L:=[L,~~]$.
For $X\in\mathfrak{g}$ denote by $X^L\in\mathfrak{g}^L$ the 
orthogonal projection of $X$ on $\mathfrak{g}^L$.  
Then define the inner product of the tangent vectors $[L,X]$ and $[L,Y]$
at $L$ to the orbit through $L$  to equal $\langle X^L,Y^L\rangle$. 
This defines a Riemannian metric on the adjoint orbit through $L$ which 
will be denoted by $\langle ~~,~~\rangle_n$ and is called the 
\textit{normal metric}.  Then we have the following result.

\begin{thm} The flow \eqref{dbb} is the gradient vector field of 
$H(L)=\kappa (L,N)$ on the adjoint orbit $\mathcal{O}$ of 
$\operatorname{U}(n)$ containing the initial condition 
$L(0)=L_0$, with respect to the normal metric $\langle~~,~~\rangle_n$.
\end{thm}

\begin{proof}
We have, by the definition of the gradient,
\begin{equation}
dH\cdot [L,\delta L] = \langle\text {grad}\,H,[L,\delta L] 
\label{gradeqn}
\rangle_n\,, 
\end{equation}
where $\cdot$ denotes the natural pairing between 1-forms and 
tangent vectors, and $[L,\delta L]$ 
is an arbitrary tangent vector  to $\mathcal{O}$
at $L$.  Set 
$\text{grad}\,H = [ L,X]$.  Then \eqref{gradeqn} becomes
$$
-\langle[ L,\delta L] ,N\rangle = \langle [L,X] ,[L,\delta L]
\rangle_n
$$

or
$$\langle [L,N],\delta L\rangle = \langle X^L,(\delta L)^L\rangle
= \langle X^L,\delta L\rangle$$
because $(\delta L)_L$ is $\left\langle \,, \right\rangle$-orthogonal 
to $X^L$. Since $\delta L$ is arbitrary and both $X^L, [L,N] \in  
\operatorname{Im} \operatorname{ad}_L$, we have $X^L=[L,N]$ and hence
$$
\text{grad}\,H =[L, [L,N]]
$$
since $X_L$ commutes with $L$.
\end{proof}

\subsection{Convexity in finite dimensions}
We now consider the Brezis equalities $\mathbf{M} =\mathbf{K} =\mathbf{D}$
from the Lie theoretic point of view presented above. We begin by recalling
the following from \cite{bloch_flaschka_ratiu93}.

Let $x=(x_1,\dots,x_n)\in\mathbb{R}^n$; then
$S_n x$ denotes the orbit of $x$ under the
symmetric group on $n$ letters, i.e., the collection of
all points $(x_{s(1)},\dots,x_{s(n)})$, where $s$ ranges 
over all $n!$ permutations. For $C\subset\mathbb{R}^n$, 
$\widehat{C}$ denotes the convex hull of $C$, i.e., the
smallest convex set containing $C$.

\begin{thm}
{\rm  Schur's Theorem}. \cite{schur23} Let $A$ be a
Hermitian matrix with eigenvalues $\lambda_j$, arranged in
non-increasing order.  Let
$\lambda=(\lambda_1,\dots,\lambda_n)$ and  $A^0=
(A_{11},\dots,A_{nn})$ be the diagonal of $A$.
Then
$$ A^0\in \widehat{S_n\lambda}.$$
\end{thm}

\begin{thm}
{\rm  Horn's Theorem}. \cite{horn54} Let $\lambda\in\mathbb{R}^n$,
with components arranged in non-increasing order. If
$A^0\in \widehat{S_n \lambda}$, there is a Hermitian matrix
$A$ with eigenvalues $\lambda$ whose diagonal is $A^0$.
\end{thm}

\begin{defn} For $x\in\mathbb{R}^n$, let
$x^*$ denote the vector obtained by rearranging the
components of $x$ in nonincreasing order. We say that $y$
{\it majorizes\/} $x$, written $x\prec y$, if
\begin{equation*}
x_1^*+\dots +x_k^*\le y_1^*+\dots y_k^*,\quad {\rm for }\quad 1\le k\le
n-1,\quad {\rm and} \quad
\sum_{j=1}^nx_j^*= \sum_{j=1}^n y_j^*.
\end{equation*}
\end{defn}

\begin{defn} An $n\times n$ real matrix
$P$ is called {\it doubly stochastic\/} if $P_{ij}\ge 0$,
and the sum of all entries in each row and each column is $1$.
\end{defn}

\begin{thm}
\label{thm_four_points_convexity}
{\rm (a)} \cite{marshall2011inequalities} $P$ is
doubly stochastic if and only if $Pe=e$ and $e'P=e'$,
where $e$ is the column vector all of whose entries are
$1$, and $e'$ is its transpose. \break 
{\rm(b)} \cite{marshall2011inequalities}
$P$ is doubly stochastic if and only if $Px\prec x$ for
all $x\in\mathbb{R}^n$. \break 
{\rm(c)} \cite{hardy1988inequalities} $x\prec y$
if and only if there is a doubly stochastic $P$ such that
$x=Py$. \break 
{\rm(d)} \cite{Bi} The set of doubly
stochastic matrices is the convex hull of its extreme
points, which are precisely the permutation matrices.
\end{thm}

\begin{rmk}
 $\{x\mid x\prec y\}= \widehat{S_n y}$.
\end{rmk}

The Schur theorem now follows easily. Diagonalize
the Hermitian matrix $A$, $A=Q\lambda Q^*$, $Q$ unitary,
$\lambda$ real diagonal.
Then $A_{ii}=\sum_j |Q_{ij}|^2\lambda_j$. If $Q$ is
unitary, the matrix $P_{ij}= |Q_{ij}|^2$ must be
doubly stochastic.  Theorem \ref{thm_four_points_convexity}(d) 
then gives the conclusion. Horn proved the converse by a rather 
intricate argument, deducing that when $x\prec\lambda$, there must 
be a doubly stochastic matrix $P$ of the form $P_{ij}=|Q_{ij}|^2$,
with $Q$ unitary, satisfying $x=P\lambda$; $Q\lambda Q^*$ is then
the desired Hermitian $A$ having eigenvalues $\lambda_j$ and
diagonal $x$.

We can see clearly from these considerations that $\mathbf{M}\ge\mathbf{K}$.

We note that problem $\mathbf{K}$ can be also be formulated in slightly 
more general form as follows, see, e.g., \cite{evans2001}.  
We allow  the cost matrix to be non-square, $n\times m$, $n\neq m$. 
In much of our analysis below we will require the square case 
which we will stipulate. 

Suppose we are given nonnegative numbers $c_{ij}$, $\mu_i^+$, $\mu_j^-$, 
$i=1,\dots, n$, $j=1,\dots, m$, satisfying 
\begin{equation}
\sum_{i=1}^n\mu_i^+=\sum_{j=1}^m\mu_j^-\,.
\end{equation}
The goal is to minimize over $\mu_{ij}\ge0$
\begin{equation}
\sum_{i=1}^n\sum_{j=1}^m\mu_{ij}c_{ij}
\end{equation}
subject to the constraints
\begin{equation}
\sum_{j=1}^m\mu_{ij}=\mu_i^+\,, \qquad \sum_{i=1}^n\mu_{ij}=\mu_J^-\,.
\end{equation}
For $m=n$  and normalizing so that the sums are unity we recover 
$\mathbf{K}$.  More generally, we can consider $m\ne n$.

\subsection{The dual problem}
We also formulate the dual problem in this setting. Maximize
\begin{equation}
\sum_{i=1}^n u_i\mu_i^+ + \sum_{j=1}^m v_j\mu_j^-
\end{equation}
subject to 
\begin{equation}
u_i+v_j\le c_{ij}\,.
\end{equation}
The dual problem  in our particular setting (where we set $m=n$) is:  maximize
\begin{equation}
\sum_{i=1}^m u_i+ \sum_{j=1}^m v_j
\end{equation}
subject to 
\begin{equation}
u_i+v_j\le  \lambda_in_j\,
\end{equation}
where $i\lambda_i$ are the eigenvalues of $L$ and $in_i$ are the 
entries of the diagonal matrix $N$.

We can clearly solve this explicitly by setting, e.g., 
$u_i=\lambda_in_i$ and $v_i=0$ or vice versa. The correct ordering 
gives the maximum and also other orderings give critical points.

For our purposes we again formulate the dual problem  
as in \cite{Brezis2018}, where $m=n$:
\begin{equation}
\mathbf{D} :=\text{sup}_{\varphi:X\rightarrow\mathbb{R},\,\psi:
Y\rightarrow\mathbb{R}}\left\{\left. \sum_{i=1}^m 
(\varphi(P_i)-\psi(N_i))\, \right|\varphi(x)-\psi(y)\le c(x,y), \,
\forall x\in X, \, y\in Y\right\}\,.
\end{equation}
It is easy to see that $\mathbf{K}\ge\mathbf{D}$, as 
described in \cite{Brezis2018}.  Indeed,
let $A =(a_{ij})$ be doubly 
stochastic as in (\ref{K}). Multiply the inequality  
$\varphi(P_i)-\psi(N_j)\le c(P_i,N_j)$ by $a_{ij}$ and sum over $i,j$:  
\begin{equation}
\sum_{j=1}^m\sum_{i=1}^m a_{ij}(\varphi(P_i) - \psi(N_j))\le 
\sum_{j=1}^m\sum_{i=1}^m a_{ij}c(P_i,N_j).
\end{equation}
Using the doubly stochastic property of $A =(a_{ij})$ in the two 
terms on the left hand side we  obtain 
\begin{equation}
\sum_{i=1}^m \varphi(P_i)-
\sum_{j=1}^m\psi(N_j) \le \sum_{j=1}^m\sum_{i=1}^m a_{ij}c(P_i,N_j),
\end{equation}
Minimizing over $A$ and maximizing over $\varphi,\psi$ 
shows that $\mathbf{K}\ge\mathbf{D}$.
\medskip

We can now see how this works in our setting.  Without loss of 
generality we choose $N$ to be diagonal. Note that $N$ for Brezis plays 
a (slightly) different role than our $N$ although they both represent points. 

\begin{thm}
$\mathbf{K}\ge\mathbf{D}$ in the orbit setting. 
\end{thm}

\begin{proof}
In our particular setting on an orbit we have $L=Q\lambda Q^*$, 
where $\lambda$  is the diagonal matrix of eigenvalues, 
$L_{ii}=\sum_j|Q_{ij}|^2\lambda_j$, and $P_i=|Q_{ij}|^2$ 
is doubly stochastic. Thus, $c(P_i,N_j)=\lambda_in_j$ and the total cost 
is $K=\sum_{i,j}P_{ij}\lambda_jn_i=\text{Trace}(LN)$ where $N$ is 
the diagonal matrix with entries $n_i$.
\end{proof}

\section{The SDiff and SMeas Setting}

Many of the ideas above generalize to an appropriate infinite-dimensional  
setting based on our work \cite{bloch_flaschka_ratiu93} on the 
infinite-dimensional Schur--Horn theorem.

Consider the following (initially smooth) setting;  
${\rm SDiff}({\mathcal A})$ is the group of
$C^\infty$ area preserving diffeomorphisms of the annulus
$$\mathcal{A}=\{0\le z \le 1\}\times \{\exp(2\pi i\theta)\mid
\, 0\le\theta<1\}$$
(more generally, one could consider Sobolev maps in $H^s$
for some $s>2$).  Its Lie algebra $\mathfrak{g}$ is identified with the
Poisson algebra of functions $x$ satisfying
$$\frac{\partial x}{\partial\t}(z_0,\t)\equiv0, \quad
z_0=0,1.$$
The Hamiltonian vector field $X_x$ will then be tangent to
the boundary.

We next consider ${\rm SMeas}({\mathcal A})$, the group of invertible measure
preserving transformations of the annulus. Each $g\in{\rm SMeas}({\mathcal A})$
determines a unitary operator $P_g$ on $L^2(\mathcal{A})$ by $P_g
x=x\circ g$. The strong operator topology induces a
topology on  ${\rm SMeas}({\mathcal A})$. It is traditionally called the 
{\it weak topology\/}, because the strong and weak operator
topologies coincide on unitary operators.

We now want to define majorization and doubly stochastic operators 
in this setting. We use the following.

\begin{defn}  \cite{hardy1988inequalities} \cite{ryff1963} Let $f\in
L^1([0,1])$. Set $m(y)=\big| \{z\mid f(z)>y\}\big|$
(absolute value denotes Lebesgue measure on $[0,1]$) and,
for $0\le z<1$, set
$$f^*(z)=\sup \{y\mid m(y)>z\}.$$ The nonincreasing, right
continuous function $f^*$ is called the {\it nonincreasing
rearrangement\/} of $f$.
\end{defn}

\begin{defn}
 Definition . \cite{hardy1988inequalities} \cite{ryff1963} Let $f,g\in
L^1([0,1])$. We say that $f$ {\it majorizes} $g$
(written $g\prec f$) if
\begin{align*}
         \int_0^s g^*(z)\,dz &\le \int_0^s
f^*(z)\,dz,\quad 0\le s<1,\\\
        \int_0^1 g^*(z)\,dz &= \int_0^1 f^*(z)\,dz.
 \end{align*}
\end{defn}

\begin{defn}  \cite{ryff1963}. A linear operator $P$
on $L^1([0,1])$ is called {\it doubly stochastic\/} if
$Pf\prec f$ for all $f\in L^1([0,1])$.
\end{defn}

\begin{thm} \cite{ryff1965}, \cite{ryff1967} In $L^1([0,1])$:
$g\prec f$ if and only if there is a doubly stochastic
$P$ such that $g=Pf$. The set $\Omega(f)= \{g\mid
g\prec f\}$ is weakly compact and convex. Its set of
extreme points is $\{f^*\circ\phi\mid \phi$ {\rm\ is\ a\
measure\ preserving\ transformation\ of\ } $[0,1]\}$.
\end{thm}

\subsection{Some spectral results}
The following results are proved in  \cite{bloch_flaschka_ratiu93} 
so we shall omit the proofs. 

\begin{thm}\label{spectral}
 Spectral Theorem. Let $x\in
L^2(\mathcal{A})\cap L^\infty(\mathcal{A})$, and set
$$I_p=\int_\mathcal{A} x^p\, dm,\quad p\in{\bf Z}^+.$$
There exists a unique, nonincreasing, right-continuous
function $\lambda$ on $[0,1]$ such that
$$I_p=\int_0^1 \lambda^p(z)\,dz,\quad p\in{\bf Z}^+.$$
\end{thm}

\begin{thm}
Diagonalization Theorem.
  Let $x\in
L^2(\mathcal{A})\cap L^\infty(\mathcal{A})$ and let $\lambda$ be as in the 
Spectral Theorem. Define $\lambda_2(z,\theta)=\lambda(z)$ There
exists a measure preserving map $\psi:\mathcal{A}\setminus{\mathcal Z}_1\to
\mathcal{A}\setminus{\mathcal Z}_2$, with ${\mathcal Z}_i$ of measure zero, such
that $x=\la_2 \circ  \psi$.
\end{thm}

\begin{thm} [Ryff]
\cite{ryff1965}, \cite{ryff1967} In $L^2([0,1])\cap
L^\infty([0,1])$: $g\prec f$ if and only if there is a
doubly stochastic operator $P$ such that $g=Pf$. The set
$\Omega(f)=\{g\mid g\prec f\}$ is weakly compact and
convex. Its set of extreme points is $\{f^*\circ\phi\mid
\phi$ {\rm\ is\ a\ measure\ preserving\ transformation\
of\ } $[0,1]\}$.
\end{thm}

We now need an analogue of the permutation group 
in our setting in order to formulate the Schur and Horn theorems.  
This corresponds to the Weyl group of the unitary group, generalized 
to our setting. For more details and the
link with the twist mapping see \cite{bloch_flaschka_ratiu93}.
The analogue of the Weyl group to our setting
is group $W$ of invertible measure preserving transformations of $[0,1]$.
The action of the the Weyl group $W$
on $\lambda$ (which is a function of $z$ alone) is just
right composition of an element of $L^2([0,1])$
by an invertible measure preserving transformations of $[0,1]$.
The Weyl semigroup $\overline{W}$ is 
the closure of $W$ in in the strong operator topology and
consists of not necessarily invertible measure preserving
transformations of $[0,1]$ (\cite[Theorem 5]{brown1966approximation}).
The action of $\overline{W}$ on $\lambda$ is again right composition.

\begin{thm}
 Schur's Theorem.  Let $x\in
L^2(\mathcal{A})\cap L^\infty(\mathcal{A})$, let $\pix$ be the zeroth
Fourier coefficient of $x$,
$$\pix(z)=\int_0^1 x(z,\t)\,d\t,$$
and let $\lambda$ be as in the spectral theorem. Then
$\pix$ belongs to the closed convex hull of the orbit of
the Weyl semigroup $\overline{W}$ through $\lambda$ 
which happens precisely when $\pix\prec\lambda$.
\end{thm}

\begin{thm}
 Horn's Theorem. Let $\lambda$ be a bounded,
nonincreasing function on $[0,1]$ and let $X$ lie in the
closed convex hull of the Weyl semigroup orbit through
$\lambda$,
$$\overline{W}\cdot\lambda=\{\la\circ\phi\mid \phi {\rm\ is\ a\
measure\ preserving\ transformation\ of\ }[0,1]\}.$$
Then there exists an $x\in L^2([0,1])\cap L^\infty([0,1])$
such thatn
\begin{align*}
(i)\quad &X(z)=\pix(z)=\int_0^1 x(z,\t)\,d\t,\\
(ii)\quad &\int_0^1\!\!\int_0^1 x(z,\t)^p\,d\t dz=
\int_0^1 \la(z)^p\,dz,\ p\in\mathbb{Z}^+.
\end{align*}
\end{thm}

\subsection{Problems \textbf{M} and \textbf{K}} We formulate these
two  problems in our setting and then prove that $\mathbf{M}$ implies 
$\mathbf{K}$. In both problems, and later in the dual problem
$\mathbf{D}$, the integrand $c(x,y)$ of the cost function is a continuous
real-valued function on $\mathbb{R}^n\times\mathbb{R}^n$.

\underline{Problem \textbf{K}}. The analogue of $\mathbf{K}$ in 
our infinite dimensional setting is the following (see, e.g., \cite{evans2001}).
Consider the class of probability measures $\mu$ on 
$\mathbb{R}^n\times\mathbb{R}^n$
with $\rm{proj}_x\mu=\mu^+\,, \rm{proj}_y\mu=\mu^-$. We wish to 
 find $\mu$ which minimizes
\begin{equation}
J[\mu]=\int_{\mathbb{R}^n\times\mathbb{R}^n}c(x,y)d\mu(x,y)\,.
\end{equation}

\underline{Problem \textbf{M}}. The analogue of $\mathbf{M}$ in 
our infinite dimensional setting is the following.
Given are two nonnegative Radon measures $\mu^+,\mu^-$ on $\mathbb{R}^n$  
satisfying $\mu^+(\mathbb{R}^n)=\mu^-(\mathbb{R}^n)$. 
Consider the class of measurable 1-1 mappings 
$s:\mathbb{R}^n\rightarrow \mathbb{R}^n$ which 
rearrange $\mu^+$ to $\mu^-$, 
$s_{\#}(\mu^+)=\mu^-$ ($s_{\#}$ denotes push forward of measures),
i.e.,
\begin{equation}
\int_X h(s(x)d\mu^+(x)=\int_Y h(s(x)d\mu^-(y)\,,
\end{equation}
for  any continuous function $h$, where 
$X$ is the support of $\mu^+$
and $Y$ is the support of $\mu^-$. We want to find $ s $ 
which minimizes
\begin{equation}
I[s]=\int_{\mathbb{R}^n}c(x,s(x))d\mu^+(x)\,.
\end{equation}
The cost function is often chosen to be quadratic: 
$$
c(x,y)=\tfrac{1}{2} \|x-y\|^2\,.
$$
This corresponds to the Wasserstein distance $W_2$. We recall its 
definition (see\cite{evans2001}).

\begin{defn}
 Let $f^+, f^-:\mathbb{R}^n \rightarrow  \mathbb{R}$ be
smooth functions with compact supports. The Wasserstein 
distance $W_2$ between them is given by 
\begin{equation}
\label{Wasserstein_distance}
d(f^+,f^-)^2=\text{inf}\left\{\frac{1}{2}\int_{\operatorname{supp}f^+} 
\int_{\operatorname{supp}f^-} 
\|x-y\|^2d\mu(x,y)\right\}
\end{equation}
where the infimum is over all nonnegative Radon measures $\mu$ with 
projections $\mu^+=f^+dx,\,\mu^-=f^-dy$, i.e., $\mu^+,\mu^-$  are 
assumed to be given by the smooth densities $f^+,f^-$.
\end{defn}

In our finite dimensional setting we can think of 
$s(\Lambda)=\Theta^T\Lambda\Theta$ and $c =\|N-L\|=
\|N-\Theta^T\Lambda\Theta\|$.

\medskip

Now we show that $\mathbf{M} \ge \mathbf{K}$.
Thus, in the context of our infinite dimensional setting, we consider 
the cost function
\begin{equation}
-\left\langle x(z,\theta),z\right\rangle =
-\int_0^1\!\!\!\int_0^1 x(z,\theta)zdzd\theta\,.
\label{cost1}
\end{equation}
We know there is a measure preserving map on $\mathcal{A}$ such 
$x=\lambda_2\circ\psi$, where we define $\lambda_2(z,\theta):=\lambda(z)$. 
This is the analogue of conjugation of a diagonal matrix by a unitary one.
Further, if  $\psi(z,\theta)=(\tau(z,\theta),\eta(z,\theta))$, then 
\begin{equation}\label{ds}
\int_0^1x(z,\theta)d\theta=\int_0^1\lambda(\tau(z,\theta))d\theta
=(P\lambda)(z)\,,
\end{equation}
where $P$ is a doubly stochastic operator. Thus we get the following result. 
\begin{thm}
$\mathbf{M}\ge\mathbf{K}$
in the setting of the diffeomorphism group of the annulus. 
\end{thm}
We can also solve the Monge problem in this case providing the minimizer.
\medskip

{\bf Solving the  Monge problem: } We note firstly that  our Schur theorem 
together with diagonalization problem above solves the Monge problem in 
this case: we have found $\psi$, a measure preserving map, that 
determines
a nondecreasing $\lambda$ with the same moments as $x$. This is the minimizer 
of the cost function for $x(z,\theta)$ paired with $z$, exactly as in the 
finite dimensional case. 
\medskip

There is also a dynamical way of understanding the optimization in the 
infinite dimensional setting, as described in \cite{bloch_flaschka_ratiu96}.
We will omit the technical details in this paper but note there is a 
version of the normal metric on the adjoint  orbit of the diffeomorphism 
group. The gradient flow of the function \eqref{cost1} with respect to the 
normal metric  in this case is 
\begin{equation}
x_t=\{x,\{x,z\}\}\,,
\end{equation} 
where the bracket in this setting is the canonical $(z,\theta)$-Poisson 
bracket. One can show that the flow converges to the minimum
$x$, is monotonically aligned with $z$, and has the same moments as 
the initial data, as long as the initial data is monotonic. Otherwise, the 
flow develops shocks. 

This is an equation of Monge--Amp\`ere type. 

\subsection{The dual problem \textbf{D}} In general, this problem 
is formulated in the following way (see \cite{evans2001}).

\underline{Problem \textbf{D}}. Given  Radon measures 
$\mu^+$, $ \mu^-$ as in Problem $\mathbf{M}$, we want to maximize
over all continuous functions $u, v:\mathbb{R}^n \rightarrow  \mathbb{R}$  
the functional   
\begin{equation}
K[u,v]=\int_X u(x)d\mu^+(x) +\int_Y v(y)d\mu^-(y),
\end{equation}
where $X$ is the support of $\mu^+$ and $Y$ is the support of $\mu^-$,
subject to $u(x)+v(y)\le c(x,y)$.
\medskip

In the diffeomorphism group setting we take   
$c:C^0([0,1], \mathbb{R}) \times [0,1] \rightarrow  
\mathbb{R}$, $c(\alpha (z),z):=-\alpha(z)z$, where  $\alpha(z)$ 
is a given measurable rearrangement of $\lambda(z)$,  $\lambda(z)$ as 
in the Spectral Theorem \ref{spectral}, and $u=u(z)$, $v=v(\alpha(z))$.
The problem is to  find the continuous functions 
$u$ and $v$ that maximize 
\begin{equation}
\int_0^1u(z)dz+\int_0^1v(\alpha (z))dz
\end{equation}
subject to 
$$
u(z)+v(\alpha(z))\le -z\alpha(z).
$$

\begin{thm}
  $\mathbf{K}\ge\mathbf{D}$ in the diffeomorphism group setting. 
\end{thm}

 \begin{proof}
 
 Integrating the inequality constraint we  have 
 \begin{equation}\label{inte}
\int_0^1u(z)dz+\int_0^1v(\alpha(z))dz\le -\int_0^1z\alpha(z)zdz
\end{equation}
The right hand side of this inequality is clearly minimized when 
$\alpha(z)=\lambda(z)$ since $\lambda(z)$ is nonincreasing. 

Now for $\mathbf{K}$ we want to minimize 
 \begin{equation}
-\left\langle x(z,\theta),z\right\rangle =
-\int_0^1\!\!\!\int_0^1 x(z,\theta)zdzd\theta\
\label{cost2}
\end{equation}
over all $x$ with fixed moments. This is also minimized when 
$x(z,\theta)=\lambda (z)$ and thus minimizing over this and 
maximizing over the left hand side of  (\ref{inte}) we obtain 
the result.
\end{proof}

\begin{rmk}
Comparison with the Wasserstein gradient: we can compare our gradient flows those  arising from the Wasserstiein metric.  Consider again the 
cost function 
\begin{equation}
-\left\langle x(z,\theta),z \right\rangle =
-\int_0^1\!\!\!\int_0^1 x(z,\theta)zdzd\theta\,.
\end{equation}
In general, the $W_2$-gradient of  $E(\rho)$ is given by \cite{jordan1997free}
\begin{equation}
\nabla _{W_2}E(\rho)=
-\nabla\cdot\left(\rho\nabla\frac{\partial E}{\partial\rho}\right)\,.
\end{equation}
If we choose $E=\int_0^1\!\!\int_0^1 z\rho(z,\theta)$ as above, 
we clearly get simply $\rho_z$ and the flow
$$
\rho_t+\rho_z=0\,.
$$
This is the simplest continuity equation in accordance with the fact that all 
Wasserstein gradient flows are continuity equations,
$$
\rho_t+\nabla\cdot(v\rho)=0\,.
$$
Here we reduced  to the one-dimensional case with constant velocity which is,
 of course, not interesting for us. 
 \end{rmk}

In future work we will consider the  extension of our infinite-dimensional 
Schur--Horn theorem to more general measure spaces 
and its application to Monge--Kantorovich in other settings. 

{\bf Acknowledgement:} We would like to thank Wilfred Gangbo,  Tryphon Georgiou and Robert McCann for useful conversations.


\begin{thebibliography}{10}
\baselineskip0.4cm

\bibitem[Bir46]{Bi}
George Birkhoff. Three observations on linear algebra. \textit{Math. 
Finance}, 5:147--151, 1946.

\bibitem[BBR92]{bloch_brockett_ratiu92}
Anthony M. Bloch, Roger W. Brockett, and Tudor S. Ratiu.
Completely integrable gradient flows. \textit{Comm. Math. Phys.}, 
147(1):57--74, 1992.

\bibitem[BFR93]{bloch_flaschka_ratiu93}
Anthony M. Bloch, Hermann Flaschka, and Tudor S. Ratiu. A 
Schur--Horn--Kostant convexity theorem for the diffeomorphism 
group of the annulus. \textit{Invent. Math.}, 113(3):511--529, 1993.

\bibitem[BFR96]{bloch_flaschka_ratiu96}
Anthony M. Bloch, Hermann Flaschka, and Tudor S. Ratiu.
The Toda PDE and the geometry of the diffeomorphism group of the 
annulus. In \textit{Mechanics Day (Waterloo, ON, 1992)}, volume 7 of 
\textit{Fields Inst. Commun.}, pages 57--92. Amer. Math. Soc., 
Providence, RI, 1996.

\bibitem[BK23]{bloch_karp23a}
Anthony M. Bloch and Steven N. Karp. Gradient flows, adjoint orbits, 
and the topology of totally nonnegative flag varieties. \textit{Comm. 
Math. Phys.}, 398(3):1213--1289, 2023.


\bibitem[Bre18]{Brezis2018}
Ha\"{\i}m Brezis. Remarks on the Monge--Kantorovich problem in the 
discrete setting. \textit{Comptes Rendus. Math\'ematique}, 
356(2):207--213, 2018.


\bibitem[Bro66]{brown1966approximation}
James Brown. Approximation theorems for Markov operators. 
\textit{Pacific Journal of Mathematics}, 16(1):13--23, 1966.

\bibitem[CGP21]{chen2021optimal}
Yongxin Chen, Tryphon T. Georgiou, and Michele Pavon. Optimal 
transport in systems and control. \textit{Annual Review of Control, 
Robotics, \& Autonomous Systems}, 4:89--113, 2021.


\bibitem[CGT17]{chen2017matrix}
Yongxin Chen, Tryphon T. Georgiou, and Allen Tannenbaum. Matrix 
optimal mass transport: a quantum mechanical approach. \textit{IEEE 
Transactions on Automatic Control}, 63(8):2612--2619, 2017. 

\bibitem[Eva]{evans2001}
Lawrence C. Evans. Partial differential equations and Monge--Kantorovich 
mass transfer. In \textit{Current Developments in Mathematics (Cambridge, 
MA, 1997)}, pages 65--126. Int. Press, Boston, MA, 1999.

\bibitem[GM96]{gangbo1996geometry}
Wilfrid Gangbo and Robert J. McCann. The geometry of optimal transportation. 
\textit{Acta Mathematica}, 177(2):113--161, 1996.

\bibitem[HLP88]{hardy1988inequalities}
Godfrey H. Hardy, John E. Littlewood, and George P\'olya. 
\textit{Inequalities}. Cambridge University Press, Cambridge, 1988.

\bibitem[Hor54]{horn54}
Alfred Horn. Doubly stochastic matrices and the diagonal of a 
rotation matrix. \textit{Amer. J. Math.}, 76:620--630, 1954.

\bibitem[JKO97]{jordan1997free}
Richard Jordan, David Kinderlehrer, and Felix Otto. Free energy and 
the Fokker-Planck equation. \textit{Physica D: Nonlinear Phenomena}, 
107(2--4):265--271, 1997.

\bibitem[MOA11]{marshall2011inequalities}
Albert W. Marshall, Ingram Olkin, and Barry C. Arnold.
\textit{Inequalities: Theory of Majorization and its Applications}. 
Second edition. Springer Series in Statistics. Springer, New York, 2011.

\bibitem[Ryf63]{ryff1963}
John V. Ryff. On the representation of doubly stochastic operators. 
\textit{Pac. J. Math.}, 13:1379--1386, 1963.

\bibitem[Ryf65]{ryff1965}
John V. Ryff. Orbits of $L^1$-functions under doubly stochastic 
transformations. \textit{Trans. Amer. Math, Soc.}, 117:92--100, 1965.

\bibitem[Ryf67]{ryff1967}
John V. Ryff. Extreme points of some convex subsets of $L^1(0,1)$.
\textit{Proc. Amer. Math, Soc.}, 18:1026--1034, 1967.



\bibitem[Sch23]{schur23}
Issai Schur. \"Uber eine Klasse von Mittelbildungen mit Anwendungen 
auf die Determinantentheorie. \textit{S.-B. Berlin. Math. Ges.}, 
22:9--20, 1923.

\end{thebibliography}
\end{document}